\documentclass[english]{article}
\usepackage[T1]{fontenc}
\usepackage[latin9]{luainputenc}
\usepackage{color}
\usepackage{units}
\usepackage{mathrsfs}
\usepackage{amsmath}
\usepackage{amsthm}
\usepackage{amssymb}
\usepackage{setspace}
\makeatletter
\theoremstyle{plain}
\newtheorem{thm}{\protect\theoremname}[section]
  \theoremstyle{plain}
  \newtheorem*{question*}{\protect\questionname}
  \theoremstyle{plain}
  \newtheorem{cor}[thm]{\protect\corollaryname}
  \theoremstyle{plain}
  \newtheorem{question}[thm]{\protect\questionname}
  \theoremstyle{remark}
  \newtheorem*{acknowledgement*}{\protect\acknowledgementname}
  \theoremstyle{plain}
  \newtheorem{lem}[thm]{\protect\lemmaname}
  \theoremstyle{definition}
  \newtheorem{defn}[thm]{\protect\definitionname}
  \theoremstyle{plain}
  \newtheorem{prop}[thm]{\protect\propositionname}
\newcommand{\lyxaddress}[1]{
\par {\raggedright #1
\vspace{1.4em}
\noindent\par}
}

\usepackage{enumitem}
\setenumerate{label=(\roman{enumi})}

\makeatother

\usepackage{babel}
  \providecommand{\acknowledgementname}{Acknowledgement}
  \providecommand{\corollaryname}{Corollary}
  \providecommand{\definitionname}{Definition}
  \providecommand{\lemmaname}{Lemma}
  \providecommand{\propositionname}{Proposition}
  \providecommand{\questionname}{Question}
\providecommand{\theoremname}{Theorem}

\begin{document}

\title{Fluctuations of ergodic averages for amenable group actions}

\author{Uri Gabor\thanks{Supported by ERC grant 306494 and ISF grant 1702/17}}
\maketitle
\begin{abstract}
We show that for any countable amenable group action, along Følner
sequences that have for any $c>1$ a two sided $c$-tempered tail,
one have universal estimate for the probability that there are $n$
fluctuations in the ergodic averages of $L^{\infty}$ functions, and
this estimate gives exponential decay in $n$. Any two-sided Følner
sequence can be thinned out to satisfy the above property, and in
particular, any countable amenable group admits such a sequence. This
extends results of S. Kalikow and B. Weiss \cite{key-1} for $\mathbb{Z}^{d}$
actions and of N. Moriakov \cite{key-2} for actions of groups with
polynomial growth.
\end{abstract}

\section{Introduction}

A real valued sequence is said to \textit{fluctuate $N$ times across
a gap $(\alpha,\beta)$}, if there are integers $n_{1}<n_{2}<...<n_{2N}$
s.t. for odd $i$, $a_{n_{i}}\leq\alpha$, and for even $i$, $a_{n_{i}}\geq\beta$.
Let $(X,\mu,\mathscr{B},(T_{g})_{g\in G})$ be a measure preserving
action of a countable amenable group $G$, and fix some (\textcolor{black}{left})
Følner sequence $(F_{n})$ in $G$. For any $N,$ we define the set
$D_{N}$ by:
\[
D_{N}=D_{(F_{n}),f,\alpha,\beta,N}=\{x\,:\,\text{\ensuremath{\mathsf{A}_{n}f(x)} fluctuates across \ensuremath{(\alpha,\beta)} at least \ensuremath{N} times}\}
\]
where $\mathsf{A}_{n}f=\frac{1}{|F_{n}|}\sum_{g\in F_{n}}f\circ T_{g}$
denotes the sequence of ergodic averages of a function $f$ on $X$
along $(F_{n})$. In \cite{key-1} it was shown that for $G=\mathbb{Z}^{d}$
and $F_{n}=[-n,n]^{d}$, the following holds:
\begin{thm}
\label{thm:Z^d fluctuation}For any $0<\alpha<\beta$, there are constants
$0<c_{0}<1$ and $c_{1}>0$, s.t. for every m.p.s. $(X,\mathscr{B},\mu,\{T_{g}\}_{g\in G})$
and every measurable $f\geq0$, one has: 
\[
\mu(D_{N})\leq c_{1}c_{0}^{N}\quad(\forall N)
\]
.
\end{thm}

In \cite{key-2} this result was extended to measure preserving actions
of groups of polynomial growth, where the fixed Følner sequence is
taken to be balls of increasing radii, that is, $F_{n}=S^{n}$ where
$S$ is a finite symmetric set of generators which contains the unit.

The aim of this paper is to extend these results to general actions
of amenable groups. In this context, the notion of temperedness is
of importance: A sequence $(F_{n})$ is \textit{left $c$-tempered}
if for all $n$, 
\[
\left|\bigcup_{i<n}F_{i}^{-1}F_{n}\right|\leq c\left|F_{n}\right|,
\]
\textit{\textcolor{black}{right $c$-tempered}} if for all $n$,
\[
\left|\bigcup_{i<n}F_{n}F_{i}^{-1}\right|\leq c\left|F_{n}\right|,
\]
and \textit{$c$-bi-tempered} if it is both left and right $c$-tempered.
In this paper, a sequence that for any $c>1$ has some tail which
is $c$-bi-tempered, will be called \textit{\textcolor{black}{strongly
tempered}}. Notice that any two-sided Følner sequence can be thinned
out to be strongly tempered.

The class of tempered Følner sequences is the most general class of
sequences which are known to satisfy the pointwise ergodic theorem
\cite{key-16,key-15}. That is, the averages along any (left) tempered
Følner sequence of any integrble function converges a.e. Consequently,
if the fixed Følner sequence $\left(F_{n}\right)$ is tempered, then
for any $\alpha<\beta$ and any integrable function $f$, the measure
of $D_{N}=D_{(F_{n}),f,\alpha,\beta,N}$ decreases to zero as $N\rightarrow\infty$.
Thus, along such sequences, one might hope to have some control on
the rate of $\mu(D_{N})$, as in Theorem \ref{thm:Z^d fluctuation}:
\begin{question*}
Does every amenable group have a Følner sequence that satisfies (in
some sense) Theorem \ref{thm:Z^d fluctuation}? Can one find for any
Følner sequence a subsequence with this property?
\end{question*}
Our main result is the following theorem and its corollary, which
says that one can successfully bound the rate of decrease of $\mu(D_{N})$
in any amenable group, provided that $f$ is bounded, and that the
averages are taken along strongly tempered Følner sequences.
\begin{thm}
\label{thm:Amn Flucts}For any $\alpha<\beta$ and $S>0$, there exist
$\lambda>0$ and $0<c_{0}<1$, s.t. for any $(1+\lambda)$-bi-tempered
Følner sequence $\left(F_{n}\right)$, any m.p.s. $(X,\mu,\mathscr{B},(T_{g})_{g\in G})$
and any $f\in L_{\mu}^{\infty}(X)$ with $||f||_{\infty}\leq S$,
one has 
\[
\mu(D_{N})\leq c_{1}c_{0}^{N}\quad(\forall N)
\]
for some $c_{1}>0$ which depends only on the sequence $\left(F_{n}\right)$
(and neither on the m.p.s. nor on the function $f$).
\end{thm}

If $\left(F_{n}\right)$ is strongly tempered, then for any gap $(\alpha,\beta)\subset\mathbb{R}$
and any $S>0$, some tail of the sequence, say $\left(F_{n}\right)_{n>n_{0}}$,
satisfies the hypothesis of Theorem \ref{thm:Amn Flucts}, while the
first $n_{0}$ elements of $\left(F_{n}\right)$ atribute at most
$O(n_{0})$ fluctuations. Thus, enlarging $c_{1}$ depending on that
$n_{0}$, we get:
\begin{cor}
\label{cor:Let--be}Let $\left(F_{N}\right)$ be a strongly tempered
Følner sequence. For any $\alpha<\beta$ and $S>0$, there exist $0<c_{0}<1$
and $c_{1}>0$, s.t. for any m.p.s. $(X,\mu,\mathscr{B},(T_{g})_{g\in G})$
and any $f\in L_{\mu}^{\infty}(X)$ with $||f||_{\infty}\leq S$,
one has 
\[
\mu(D_{N})\leq c_{1}c_{0}^{N}\quad(\forall N)
\]
\end{cor}

As the proof of Theorem \ref{thm:Amn Flucts} indicates, the bi-temperedness
condition could be slightly relaxed, and was chosen for the clarity
of presentation. In addition, the dependency of $c_{1}$ on the sequence
$\left(F_{n}\right)$ could be replaced by restricting the theorem
to sequences with some certain properties. For example, assuming $e\in F_{1}$
would be enough for determine $c_{1}$, regardless of what $\left(F_{n}\right)$
is.

In contrast to Theorem \ref{thm:Amn Flucts}, we show that the temperedness
property (with any fixed $c>1$) alone, isn't enough to bound the
rate of decrease of $\mu(D_{N})$ for any given gap $(\alpha,\beta)$.
More precisely, we show that in any measure preaserving $\mathbb{Z}$-action
$(X,\mu,\mathscr{B},\{T^{n}\}_{n\in\mathbb{Z}})$ one has the following:
\begin{thm}
\label{prop:counterexample}Let $(X,\mathscr{B},\mu,\{T^{n}\}_{n\in\mathbb{Z}})$
be a m.p.s. and let $\omega(n)\searrow0$ be any sequence which decreases
to 0. For any $\lambda>0$, there are some $\alpha<\beta$, a bounded
function $0\leq f\leq1$ and a $(1+\lambda)$-tempered Følner sequence
$\left(F_{n}\right)$, for which $\mu(D_{(F_{n}),f,\alpha,\beta,N})>\omega(N)$
for all but finitely many $N$.
\end{thm}

Although this shows that the requirement for $\left(F_{n}\right)$
to have a left $(1+\lambda)$-tempered tail for any $\lambda$ is
essential for Corollary \ref{cor:Let--be} to take place, it is not
clear whether the other requirements are. More generally, the following
question remains open:
\begin{question}
Does every left Følner sequence in a countable amenable group $G$
have a subsequence which satisfies the conclusion of Corollary \ref{cor:Let--be}?
\end{question}

\begin{acknowledgement*}
I would like to thank my advisor Michael Hochman, for suggesting me
the problem studied in this paper, and for many helpful discussions.
\end{acknowledgement*}

\section{Proof of Theorem \ref{prop:counterexample}}

\textbf{Proof of Theorem \ref{prop:counterexample}:} Let $\lambda>0$.
We first construct finite sequences of subsets of $\mathbb{Z}$, which
have good fluctuation and invariance properties, and then concatenate
such sequences to get the whole sequence $\left(F_{n}\right)$ in
question. Fix $l,N\in\mathbb{N}$, and let $\phi_{l}:\mathbb{N}_{0}\rightarrow\{0,1\}$
be the indicator function 
\[
\phi_{l}=\mathbf{1}_{(2l\mathbb{N}_{0}+[0,l-1])}
\]
(here $\mathbb{N}_{0}=\mathbb{N}\cup\{0\}$.) We define a sequence
of subsets $\left(A_{n}\right)_{n=1}^{2N}=\left(A_{n}^{l,N}\right)_{n=1}^{2N}$
recursively:
\[
A_{n+1}=\begin{cases}
[0,\frac{2}{\lambda}M_{n}]\cup\left((2l\mathbb{N}_{0}+[0,l-1])\cap[0,\frac{2+\lambda}{\lambda}M_{n}]\right) & \text{\ensuremath{n+1} is odd}\\{}
[0,\frac{2}{\lambda}M_{n}]\cup\left((2l\mathbb{N}_{0}+[l,2l-1])\cap[0,\frac{2+\lambda}{\lambda}M_{n}-2l+1]\right) & \text{\ensuremath{n+1} is even}
\end{cases}
\]
where $M_{0}=l^{2}$, and $M_{n}=\max(A_{n})$ for $n>0$. This sequence
has the following properties: 

\textbf{(a)} $\left(A_{n}\right)_{n=1}^{2N}$ is $(1+\lambda)$-tempered:
For any $n$, 
\[
A_{n}-\bigcup_{i=0}^{n-1}A_{i}\subset A_{n}-[0,M_{n-1}]\subset[-M_{n-1},M_{n}]
\]
thus 
\[
|A_{n}-\bigcup_{i=0}^{n-1}A_{i}|\leq M_{n}+M_{n-1}\leq\frac{2+\lambda}{\lambda}M_{n-1}+M_{n-1}\leq(1+\lambda)|A_{n}|.
\]

\textbf{(b)} $A_{n}$ is $([-\sqrt{l},\sqrt{l}],2/\sqrt{l})$-invariant
for all $n$; that is, for any $b\in[-\sqrt{l},\sqrt{l}]$, one have
$\frac{|(b+A_{n})\triangle A_{n}|}{|A_{n}|}\leq2/\sqrt{l}$: This
follows immidiately from the fact that $A_{n}$ is a union of segments,
the first one of size at list $l^{2}$, and all but the last one of
size at least $l$.

\textbf{(c)} Assuming $l$ large enough, there are some $0<\alpha<\beta$
s.t. for any $0\leq i\leq l/4$, and any $k$, averaging $\phi_{l}(z+2lk+i)$
as a function of $z$ along $A_{n}$, the sequence of averages fluctuates
across the gap $(\alpha,\beta)$ $N$ times: Averaging along $A_{n}$
of odd $n$ gives
\begin{align*}
\frac{1}{|A_{n}|}\sum_{z\in A_{n}}\phi_{l}(z+2lk+i) & =\frac{\left|(2l\mathbb{N}_{0}+[0,l-1]-i)\cap A_{n}\right|}{|A_{n}|}\\
 & \geq\frac{1}{\left|A_{n}\right|}\cdot\left[l\cdot\left(\frac{\frac{2}{\lambda}M_{n-1}}{2l}-1\right)+\frac{3}{4}l\cdot\left(\frac{M_{n-1}}{2l}-1\right)\right]\\
 & \geq\frac{\left(\frac{1}{\lambda}+\frac{3}{8}\right)M_{n-1}-2l}{\left(\frac{2}{\lambda}+\frac{1}{2}\right)M_{n-1}+l}\\
 & \geq\frac{1}{2}+\frac{\lambda}{4(4+\lambda)}-\frac{4}{l}
\end{align*}
while for even $n$,
\begin{align*}
\frac{1}{|A_{n}|}\sum_{z\in A_{n}}\phi_{l}(z+2lk+i) & \leq\frac{1}{\left|A_{n}\right|}\cdot\left[l\cdot\left(\frac{\frac{2}{\lambda}M_{n-1}}{2l}+1\right)+\frac{1}{4}l\cdot\left(\frac{M_{n-1}}{2l}+1\right)\right]\\
 & \leq\frac{\left(\frac{1}{\lambda}+\frac{1}{8}\right)M_{n-1}+2l}{\left(\frac{2}{\lambda}+\frac{1}{2}\right)M_{n-1}-l}\\
 & \leq\frac{1}{2}-\frac{\lambda}{4(4+\lambda)}+\frac{4}{l}
\end{align*}
(for the error summands $\frac{4}{l}$, we used $M_{n-1}\geq M_{0}=l^{2}$
and assumed $l\geq4$.) Taking $l$ large enough, one get that the
claim above takes place with $\alpha=\frac{1}{2}-\frac{\lambda}{5(4+\lambda)}$
and $\beta=\frac{1}{2}+\frac{\lambda}{5(4+\lambda)}$.

Now construct the whole sequence as follows: Take $l_{0}>100$ and
also large enough so that property (c) takes place, and then define
$(l_{m})_{m=1}^{\infty}$ recursively by the rule
\[
l_{k+1}=\max(A_{2l_{m}}^{l_{m},2l_{m}}).
\]

We define $(F_{n})$ to be the concatenation of the sequences $\{(A_{n}^{l_{m},2l_{m}})_{n=1}^{2l_{m}}\}_{m=0}^{\infty}$.
Using properties (a) and (b) above together with the definition of
$(l_{m})_{m=1}^{\infty}$, one can observe that this sequence is a
$(1+\lambda)$-tempered Følner sequence.

To construct the function which satisfies the conclusion of the theorem,
we need some notation which will be of use here and in the rest of
the paper: For a given function $f$ on a m.p.s. $(X,\mathscr{B},\mu,\{T_{g}\}_{g\in G})$,
a gap $(\alpha,\beta)\subset\mathbb{R}$, a sequence $\left(F_{n}\right)$
of subsets of $G$, and $N,M\in\mathbb{N}$, we shall write 
\begin{align*}
D_{N} & =\{x\,:\,\text{\ensuremath{\left\{ \mathsf{A}_{n}f(x)\right\} _{n=1}^{\infty}} fluctuates across \ensuremath{(\alpha,\beta)} at least \ensuremath{N} times}\}\\
D_{N,M} & =\{x\,:\,\text{\ensuremath{\left\{ \mathsf{A}_{n}f(x)\right\} _{n=1}^{M}} fluctuates across \ensuremath{(\alpha,\beta)} at least \ensuremath{N} times}\}
\end{align*}
where the sequence $\left(F_{n}\right)$, the function $f$ and the
gap $(\alpha,\beta)$ are understood from the context. At some places
we shall write $D_{N}^{f}$ and $D_{N,M}^{f}$ to specify the function
$f$ for which the sets refer to.

We will construct the function in question by applying iteratively
infinitely many times the following lemma: 
\begin{lem}
\label{lem:Let-.-For}Let $f:X\rightarrow[0,1]$. For any $\epsilon>0$,
$1>\delta>0$, and $n',N',N''\in\mathbb{N}$, there exists a measurable
function $\hat{f}:X\rightarrow[0,1]$ s.t. the following holds:
\end{lem}

\begin{enumerate}
\item \label{enu:,}$\mu\left(D_{N''}^{\hat{f}}\right)>\frac{1}{10}\delta$.
\item \label{enu:,-where-,}$\mu\left(\left(f(T^{i}x)\right)_{i=0}^{L-1}\neq\left(\hat{f}(T^{i}x)\right)_{i=0}^{L-1}\right)\leq\delta$,
where $L:=\max\left(\bigcup_{n=1}^{n'}F_{n}\right)$.
\item \label{enu:For-all-,}For all $N\leq N'$, $\mu\left(D_{N}^{\hat{f}}\right)\geq\min\left\{ \mu(D_{N,n'}^{f})-\epsilon,\frac{1}{10}\right\} $.
\end{enumerate}
\begin{proof}
We will assume w.l.o.g. that $\epsilon$ is small enough so that $\epsilon<\min\left\{ \frac{\delta}{100},1-\delta\right\} $.
Take an $m\in\mathbb{N}$ that satisfies $l_{m}\geq N''$. Let $B\subset X$
be a base for a Rokhlin tower of height $h$ and total measure $>1-\epsilon/4$,
where $h$ is large enough to satisfy
\begin{equation}
h>\frac{(L+\max A_{2l_{m}}^{l_{m},2l_{m}})}{\epsilon/4}.\label{eq:-1-1}
\end{equation}
and also large enough to guarantee that
\begin{equation}
\mu\left(x\in B:\,\forall N\leq N',\sum_{i=0}^{h-1}\mathbf{1}_{D_{N,n'}^{f}}(T^{i}x)>\mu(D_{N,n'}^{f})-\epsilon/4\right)>(1-\epsilon/4)\mu(B)\label{eq:-7}
\end{equation}
in words, for all $N\leq N'$, for at least $1-\epsilon/4$ of the
$x$'s in $B$, their orbit along the tower spends more than $\mu(D_{N,n'}^{f})-\epsilon/4$
of the time in the set $D_{N,n'}^{f}$ (For the validity of such a
requirement, see for example \cite[Theorem 7.13]{key-17}).

Take $B'\subset B$ of measure $\mu(B')=0.99\delta/h$ (this can be
achieved because $1-\epsilon>\delta$), and define $\hat{f}$ to be:
\[
\hat{f}(x)=\begin{cases}
\phi_{l_{m}}(i) & x\in T^{i}B',\,0\leq i\leq h-1\\
f(x) & x\in X\backslash\bigcup_{i=0}^{h-1}T^{i}B'
\end{cases}
\]
The validity of property (c) above for $\left(A_{n}^{l_{m},2l_{m}}\right)_{n=1}^{2l_{m}}$
and thus for $\left(F_{n}\right)$, together with the definition of
$\hat{f}$ as $\phi_{l_{m}}$ on the tower above $B'$, implies that
for any $0\leq i\leq l_{m}/4$ and any $k\geq0$ s.t. $2kl_{m}+i<h-\max A_{l_{m}}^{l_{m},2l_{m}}$,
one has:
\[
T^{2kl_{m}+i}B'\subset D_{l_{m}}^{\hat{f}}\subset D_{N''}^{\hat{f}}
\]

The density of these levels in the tower is at least 
\[
\left(\frac{l_{m}}{4}-1\right)\cdot\left(\frac{h}{2l_{m}}-1\right)/h>\frac{1}{8}-\frac{l_{m}}{h}-\frac{1}{l_{m}}
\]
and since $l_{m}\geq l_{0}>100$ and $\frac{l_{m}}{h}\leq\frac{\max A_{l_{m}}^{l_{m},2l_{m}}}{h}\leq\epsilon/4<\frac{1}{100}$,
the last expression is at least $\frac{1}{9}$. Thus
\begin{equation}
\mu\left(D_{N''}^{\hat{f}}\right)\geq\frac{1}{9}\mu\left(\bigcup_{n=0}^{h-L}T^{n}B'\right)=\frac{1}{9}(0.99\delta)>\frac{1}{10}\delta\label{eq:-3-1}
\end{equation}
 which gives property \ref{enu:,} of the conclusion.

To see why property \ref{enu:,-where-,} of the conclusion holds,
notice that
\begin{equation}
\bigcup_{n=0}^{h-L}T^{n}\left(B\backslash B'\right)\subset\left\{ x:\,\left(f(T^{i}x)\right)_{i=0}^{L-1}=\left(\hat{f}(T^{i}x)\right)_{i=0}^{L-1}\right\} \label{eq:-2-1}
\end{equation}
thus
\begin{align*}
\mu\left(\left(f(T^{i}x)\right)_{i=0}^{L-1}=\left(\hat{f}(T^{i}x)\right)_{i=0}^{L-1}\right) & \geq1-\epsilon/4-L\mu(B)-h\mu(B')\\
 & >1-\epsilon/4-\epsilon/4-0.99\delta\\
 & >1-\delta.
\end{align*}
Finally, by (\ref{eq:-2-1}) we have for all $N$
\[
\bigcup_{n=0}^{h-L}T^{n}\left(B\backslash B'\right)\cap D_{N,n'}^{f}\subset D_{N,n'}^{\hat{f}}
\]
and by (\ref{eq:-7}) and (\ref{eq:-1-1}), we have for all $N\leq N'$,
\[
\mu\left(\bigcup_{n=0}^{h-L}T^{n}\left(B\backslash B'\right)\cap D_{N,n'}^{f}\right)\geq\mu\left(D_{N,n'}^{f}\right)\mu\left(\bigcup_{n=0}^{h-L}T^{n}\left(B\backslash B'\right)\right)-\frac{3}{4}\epsilon
\]
That, together with the first inequality in (\ref{eq:-3-1}) gives
for all $N\leq N'$
\begin{align*}
\mu\left(D_{N}^{\hat{f}}\right) & \geq\mu\left(\bigcup_{n=0}^{h-L}T^{n}\left(B\backslash B'\right)\cap D_{N}^{\hat{f}}\right)+\mu\left(\bigcup_{n=0}^{h-L}T^{n}B'\cap D_{N}^{\hat{f}}\right)\\
 & \geq\mu\left(D_{N,n'}^{f}\right)\mu\left(\bigcup_{n=0}^{h-L}T^{n}\left(B\backslash B'\right)\right)-\frac{3}{4}\epsilon+\frac{1}{9}\mu\left(\bigcup_{n=0}^{h-L}T^{n}B'\right)\\
 & \geq\min\left\{ \mu\left(D_{N,n'}^{f}\right),\frac{1}{9}\right\} \mu\left(\bigcup_{n=0}^{h-L}T^{n}B\right)-\frac{3}{4}\epsilon\\
 & \geq\min\left\{ \mu\left(D_{N,n'}^{f}\right),\frac{1}{9}\right\} -\epsilon\\
 & \geq\min\left\{ \mu(D_{N,n'}^{f})-\epsilon,\frac{1}{10}\right\} 
\end{align*}
which gives property \ref{enu:For-all-,} of the conclusion.
\end{proof}
Let $\omega(n)\searrow0$ be any sequence which decreases to 0. Define
$\left(N_{k}\right)_{k=1}^{\infty}$ by 
\[
N_{k}=\min\{N:\:\omega(N)<\frac{1}{10}2^{-k-1}\}
\]
 We will construct a function $f$ which satisfies for all $k$
\[
\mu(D_{N_{k}}^{f})\geq\frac{1}{10}2^{-k}
\]
and by monotonicity of $\mu\left(D_{N}^{f}\right)$ and $\omega(N)$,
for any $N_{k}\leq N<N_{k+1}$, $k\geq1$,
\[
\mu\left(D_{N}^{f}\right)\geq\mu\left(D_{N_{k+1}}^{f}\right)\geq\frac{1}{10}2^{-k-1}>\omega(N_{k})\geq\omega(N)
\]
and the conclusion of Theorem \ref{prop:counterexample} follows.

Take $f_{0}\equiv0$, and define inductively $\left(f_{k}\right)_{k=0}^{\infty}$
: Given $f_{k-1}$, assume that
\begin{equation}
\exists n_{k-1}\,\forall1\leq i\leq k-1,\qquad\mu\left(D_{N_{i},n_{k-1}}^{f_{k-1}}\right)>\frac{1}{10}2^{-i}\label{eq:-10}
\end{equation}

Take $\epsilon>0$ small enough so that for all $i\leq k-1$,
\[
\mu\left(D_{N_{i},n_{k-1}}^{f_{k-1}}\right)-\epsilon>\frac{1}{10}2^{-i}
\]
 and apply Lemma \ref{lem:Let-.-For} with $f:=f_{k-1}$,$N'=N_{k-1}$,
$N''=N_{k}$, $n'=n_{k-1}$, $\delta=2^{-k}$ while letting $f_{k}$
be the resulting finction $\hat{f}$. This $f_{k}$ satisfies the
hypothesis (\ref{eq:-10}) in the inductive step: By property \ref{enu:For-all-,}
of the lemma, for all $i\leq k-1$, 
\begin{equation}
\mu\left(D_{N_{i}}^{f_{k}}\right)\geq\min\left\{ \mu(D_{N_{i},n_{k-1}}^{f_{k-1}})-\epsilon,\frac{1}{10}\right\} >\frac{1}{10}2^{-i}\label{eq:-4-1}
\end{equation}
 and by property \ref{enu:,} of the lemma,
\begin{equation}
\mu\left(D_{N_{k}}^{f_{k}}\right)>\frac{1}{10}2^{-k}.\label{eq:-5-1}
\end{equation}
 Since $\mu(D_{N,n}^{f_{k}})\overset{{\scriptscriptstyle n\rightarrow\infty}}{\longrightarrow}\mu(D_{N}^{f_{k}})$,
there exists large enough $n_{k}$ s.t. (\ref{eq:-4-1}) and (\ref{eq:-5-1})
will be satisfied with $D_{N,n_{k}}^{f_{k}}$ in place of $D_{N}^{f_{k}}$.
Thus the hypothesis (\ref{eq:-10}) of the induction step is indeed
satisfied with $k$ in place of $k-1$. 

We end up with a sequence $\left(f_{k}\right)_{k=0}^{\infty}$ together
with a sequence $\left(n_{k}\right)$ which we can assume to be increasing.
By property (\ref{enu:,-where-,}) of the lemma, $\left(f_{k}\right)$
converges a.e. to some limit, call it $f$. For each $k$, let 
\[
L_{k}:=\max\left(\bigcup_{n=1}^{n_{k}}F_{n}\right)
\]
then again by property (\ref{enu:,-where-,}) of the lemma, $f$ satisfies
\begin{align*}
\mu\left(\left(f_{k}(T^{n}x)\right)_{n=0}^{L_{k}-1}\neq\left(f(T^{n}x)\right)_{n=0}^{L_{k}-1}\right) & \leq\sum_{i\geq k}\mu\left(\left(f_{i}(T^{n}x)\right)_{n=0}^{L_{k}-1}\neq\left(f_{i+1}(T^{n}x)\right)_{n=0}^{L_{k}-1}\right)\\
 & \leq\sum_{i\geq k}\mu\left(\left(f_{i}(T^{n}x)\right)_{n=0}^{L_{i}-1}\neq\left(f_{i+1}(T^{n}x)\right)_{n=0}^{L_{i}-1}\right)\\
 & \leq\sum_{i\geq k}2^{-i}\\
 & =2^{-k+1}
\end{align*}
(in the second inequality we used the assumption that $n_{i}\geq n_{i-1}$
for all $i$). Thus for any $i$,
\begin{align*}
\mu\left(D_{N_{i}}^{f}\right) & \geq\mu\left(D_{N_{i},n_{k}}^{f}\right)\\
 & \geq\mu\left(D_{N_{i},n_{k}}^{f_{k}}\right)-2^{-k+1}\\
 & >\frac{1}{10}2^{-i}-2^{-k+1}
\end{align*}
taking $k\rightarrow\infty$ gives 
\[
\mu\left(D_{N_{i}}^{f}\right)\geq\frac{1}{10}2^{-i}
\]
and the proof of Theorem \ref{prop:counterexample} is complete.$\hfill\square$

\section{Proof of Theorem \ref{thm:Amn Flucts}}

In this section, we will prove Theorem \ref{thm:Amn Flucts}. Towards
this end, we need few definitions and lemmas.
\begin{defn}
\label{def:inv. cond.}Given $0<\lambda<1$, we say that a sequence
$\left(F_{n}\right)$\textit{ }is\textit{ $\lambda$-good}, if the
following two conditions hold:
\end{defn}

\begin{enumerate}
\item \label{enu:For-any-,-1}For any $n$, $\left|\bigcup_{i<n}F_{i}^{-1}F_{n}\,\backslash\,F_{n}\right|\leq\lambda\left|F_{n}\right|$.
\item \label{enu:For-any-}For any $i<n$ and $f\in F_{i}$, $\left|F_{n}\backslash F_{n}f\right|<\lambda\left|F_{n}\right|$.
\end{enumerate}
\begin{prop}
\label{prop:Let-.-For}Let $0<\lambda<\lambda'<1$. For any $(1+\lambda)$-bi-tempered
two-sided Følner sequence $\left(F_{n}\right)$, there is some $n_{0}$
s.t. $\left(F_{n}\right)_{n\geq n_{0}}$ is $\lambda'$-good.
\end{prop}

\begin{proof}
Pick some $g_{0}\in F_{1}$. Since the sequence is (left) Følner,
there is some $n_{1}$ s.t. for all $n\geq n_{1}$,
\[
|g_{0}^{-1}F_{n}\cap F_{n}|>(1-\lambda'+\lambda)|F_{n}|
\]
By the (left) temperedness property of $\left(F_{n}\right)$, we have
\begin{align*}
\left|\bigcup_{n_{1}\leq i<n}F_{i}^{-1}F_{n}\,\backslash\,F_{n}\right| & \leq\left|\bigcup_{i<n}F_{i}^{-1}F_{n}\right|-\left|\bigcup_{i<n}F_{i}^{-1}F_{n}\,\cap\,F_{n}\right|\\
 & <(1+\lambda)|F_{n}|-(1-\lambda'+\lambda)|F_{n}|\\
 & =\lambda'|F_{n}|
\end{align*}
and \ref{enu:For-any-,-1} of Definition \ref{def:inv. cond.} takes
place. The same proof applies from the right, thus we get some $n_{2}$
s.t. for any $n\geq n_{2}$
\[
\left|\bigcup_{n_{2}\leq i<n}F_{n}F_{i}^{-1}\,\backslash\,F_{n}\right|\leq\lambda'\left|F_{n}\right|
\]
but now, for any $i<n$ and $f\in F_{i}$,
\[
\left|F_{n}\backslash F_{n}f\right|=\left|F_{n}f^{-1}\backslash F_{n}\right|\leq\left|\bigcup_{n_{2}\leq i<n}F_{n}F_{i}^{-1}\,\backslash\,F_{n}\right|\leq\lambda'\left|F_{n}\right|
\]
which is \ref{enu:For-any-} of Definition \ref{def:inv. cond.}.
Now take $n_{0}=max\left\{ n_{1},n_{2}\right\} $.
\end{proof}
The following theorem is a version of Theorem \ref{thm:Amn Flucts}
for $\lambda$-good Følner sequences, from which we will deduce Theorem
\ref{thm:Amn Flucts}:
\begin{thm}
\label{thm:For-any-}For any $\alpha<\beta$ and $S>0$, there exist
$\lambda>0$, $0<c_{0}<1$ and $c_{1}>0$, s.t. for any $\lambda$-good
(left) Følner sequence $\left(F_{n}\right)$, any m.p.s. $(X,\mu,\mathscr{B},(T_{g})_{g\in G})$
and any $f\in L_{\mu}^{\infty}(X)$ with $||f||_{\infty}\leq S$,
one has 
\[
\mu(D_{N})\leq c_{1}c_{0}^{N}\quad(\forall N)
\]
\end{thm}

We remark that as opposed to Theorem \ref{thm:Amn Flucts}, here the
constant $c_{0}$ doesn't depend on $\left(F_{n}\right)$.

Once Theorem \ref{thm:For-any-} is valid, the proof of Theorem \ref{thm:Amn Flucts}
is immidiate:

\textbf{Proof of Theorem \ref{thm:Amn Flucts}:} For $[\alpha,\beta]$
and $S>0$, let $\lambda'$ be the value for which any $\lambda'$-good
Følner sequence satisfies the conclusion of Theorem \ref{thm:For-any-}
with $c_{0}'$ and $c_{1}'$. Take $0<\lambda<\lambda'$, then by
Proposition \ref{prop:Let-.-For}, for any $(1+\lambda)$-bi-tempered
two-sided Følner sequence $\left(F_{n}\right)$, there is some $n_{0}$,
s.t. $\left(F_{n}\right)_{n\geq n_{0}}$ is $\lambda'$-good, and
thus for any m.p.s. $(X,\mu,\mathscr{B},(T_{g})_{g\in G})$, any $f\in L_{\mu}^{\infty}(X)$
with $||f||_{\infty}\leq S$ and any $N$,
\[
\mu\left(D_{\left(F_{n}\right)_{n\geq1},N}\right)\leq\mu\left(D_{\left(F_{n}\right)_{n\geq n_{0}},N-n_{0}}\right)\leq c_{1}'c_{0}'^{N-n_{0}}
\]
thus for $c_{0}=c_{0}'$ , $c_{1}=c'c_{0}'^{-n_{0}}$ the conclusion
follows.$\hfill\square$

Thus it remains to prove Theorem \ref{thm:For-any-}, which will be
our task for the rest of the paper.
\begin{defn}
Given $\epsilon>0$, a collection $\left(F_{j}\right)_{j=1}^{L}$
of finite subsets of $G$ is said to be \textit{$\epsilon$-disjoint}
if there are pairwise disjoint sets $E_{j}\subset F_{j}$ s.t. $\left|E_{j}\right|\geq(1-\epsilon)\left|F_{j}\right|$
for all $1\leq j\leq L$.
\end{defn}

We record here a version of the $\epsilon$-disjointification lemma
\cite[Lemma 9.2]{key-15}, which will be uses again and again:
\begin{lem}
($\epsilon$-disjointification lemma) Let $F_{1},...,F_{L}$ be a
sequence of finite subsets of a group $G$ which is $2$-tempered,
let $C\subset G$ be finite, and suppose that $C_{1},...,C_{L}$ are
disjoint subsets of $C$. For any $0<\epsilon\leq\frac{1}{2}$, there
are subsets $D_{j}\subset C_{j}$, s.t. :
\end{lem}

\begin{enumerate}
\item The collection $\left\{ F_{j}d:\,d\in D_{j},\,1\leq j\leq L\right\} $
is $\epsilon$-disjoint,
\item $\left|\bigcup_{j=1}^{L}F_{j}D_{j}\right|\geq\frac{\epsilon}{5}\left|C\right|$.
\end{enumerate}
The following proposition, which is analogous to the effective Vitali
covering argument of Kalikow and Weiss \cite{key-1}, will be used
as a key step through.
\begin{prop}
\label{lem:Key1 Lemma} For any $\epsilon>0$, once $\lambda>0$ is
small enough and $q\in\mathbb{N}$ is large enough, the following
holds for any $\lambda$-good Følner sequence $(F_{n})$:

Let $C\subset G$ be a finite subset, and suppose that for each $c\in C$
there is associated a subsequence of $(F_{n}c)$ of length $q$: 
\[
F_{n_{1}(c)}c,...,F_{n_{q}(c)}c,\qquad n_{1}(c)<...<n_{q}(c).
\]
 Then there exists an $\epsilon$-disjoint collection $\{F_{n(d)}d\}_{d\in D}$
where $D\subset C$ and $n(d)\in\left\{ n_{1}(d),...,n_{q}(d)\right\} $
, which satisfies at least one of the following properties:

1. Either $\left|\bigcup_{d\in D}F_{n(d)}d\right|\geq2|C|$,

2. or $\left|\bigcup_{d\in D}F_{n(d)}d\cap C\right|\geq(1-\epsilon)|C|$.
\end{prop}

As it can be seen from the proof below, for $\left(F_{n}\right)$
to satisfy the conclusion, one can assume that $\left(F_{n}\right)$
is a Følner sequence that merely admits property (i) of being $\lambda$-good
(Definition \ref{def:inv. cond.}). 
\begin{proof}
Define
\[
\mathscr{C}=\left\{ (c,n_{i}(c)):c\in C,1\leq i\leq q\right\} 
\]
let $m=\max\{n:\,\exists c\in G,\,(c,n)\in\mathscr{C}\}$, and consider
the $m$-section of $\mathscr{\ensuremath{C}}$: 
\[
C_{m}=\{c:(c,m)\in\mathscr{C}\}
\]
Assuming $\lambda\leq1$, the $\epsilon$-disjointification lemma
guarantees there is a subset $D_{m}\subset C_{m}$, s.t.

(a) The collection $\{F_{m}d\}_{d\in D_{m}}$ is $\epsilon$-disjoint,
and 

(b) $|F_{m}D_{m}|\geq\frac{\epsilon}{5}\left|C_{m}\right|$.

Let $1\leq k\leq m-1$, and suppose we have already defined subsets
$\left(D_{m-i}\right)_{i=0}^{k-1}$ of $C$. Denote:
\begin{align}
W_{m-k+1} & =C\backslash\,\bigcup_{n=m-k+1}^{m}\,\bigcup_{i<n}F_{i}^{-1}F_{n}D_{n}\label{eq:-11}\\
C_{m-k} & =\left\{ c\in W_{m-k+1}:\,(c,m-k)\in\mathscr{C}\right\} \label{eq:-12}
\end{align}
and use again the $\epsilon$-disjointification lemma to take some
$D_{m-k}\subset C_{m-k}$ so that:

(a)' The collection $\{F_{m-k}d\}_{d\in D_{m-k}}$ is $\epsilon$-disjoint,
and

(b)' $|F_{m-k}D_{m-k}|\geq\frac{\epsilon}{5}\left|C_{m-k}\right|$.

The restriction $C_{m-k}\subset W_{m-k+1}$ (\ref{eq:-12}) together
with (\ref{eq:-11}) guarantees that 
\[
\bigcup_{n=m-k+1}^{m}F_{n}D_{n}\cap F_{m-k}D_{m-k}=\emptyset.
\]

We end up (after $m$ steps) with a pairwise disjoint subsets $D_{1}\subset C_{1},...,D_{m}\subset C_{m}$
where $\bigsqcup_{n=1}^{m}C_{n}\times\{n\}\subset\mathscr{C}$, and
s.t. each $\{F_{n}d\}_{d\in D_{n}}$ is $\epsilon-$disjoint, the
unions $\bigcup_{d\in D_{n}}F_{n}d=F_{n}D_{n}$ are disjoint to each
other and are of size $\left|F_{n}D_{n}\right|\geq\frac{\epsilon}{5}\left|C_{n}\right|$.
Let $\mathscr{D}=\bigsqcup_{n=1}^{m}D_{n}\times\{n\}$. We claim that
the collection $\{F_{n}d\}_{(d,n)\in\mathscr{D}}$ satisfies the conclusion
of the Lemma: We just pointed out that it is indeed an $\epsilon$-disjoint
collection. Suppose it doesn't satisfy property 2 of the conclusion,
that is,
\begin{equation}
\left|C\backslash\bigcup_{k=1}^{m}F_{k}D_{k}\right|\geq\epsilon\left|C\right|.\label{eq:}
\end{equation}
We distinguish between two cases: 

I. One has:
\[
2\lambda\left|\bigcup_{k=1}^{m}F_{k}D_{k}\right|\geq\frac{1}{2}\left|C\backslash\bigcup_{k=1}^{m}F_{k}D_{k}\right|
\]
then, together with (\ref{eq:}) one get:
\[
\left|\bigcup_{k=1}^{m}F_{k}D_{k}\right|\geq\frac{1}{4\lambda}\left|C\backslash\bigcup_{k=1}^{m}F_{k}D_{k}\right|\geq\frac{\epsilon}{4\lambda}\left|C\right|
\]
and for small enough $\lambda$ ($\lambda\leq\frac{\epsilon}{8}$),
the last inequality gives property 1 in the conclusion, so we're done.

II. For the other case,
\begin{equation}
2\lambda\left|\bigcup_{k=1}^{m}F_{k}D_{k}\right|<\frac{1}{2}\left|C\backslash\bigcup_{k=1}^{m}F_{k}D_{k}\right|\label{eq:-1}
\end{equation}
we bound from below the size of $W_{2}=C\backslash\,\bigcup_{n=2}^{m}\,\bigcup_{i<n}F_{i}^{-1}F_{n}D_{n}$
:
\begin{align*}
\left|W_{2}\right| & \geq\left|C\backslash\bigcup_{k=1}^{m}F_{k}D_{k}\right|-\left|\bigcup_{(d,n)\in\mathscr{D}}\,\bigcup_{i<n}F_{i}^{-1}F_{n}d\,\backslash\,F_{n}d\right|\\
 & \geq\left|C\backslash\bigcup_{k=1}^{m}F_{k}D_{k}\right|-\lambda\sum_{(d,n)\in\mathscr{D}}\left|F_{n}d\right|\\
 & \geq\left|C\backslash\bigcup_{k=1}^{m}F_{k}D_{k}\right|-\frac{\lambda}{1-\epsilon}\left|\bigcup_{k=1}^{m}F_{k}D_{k}\right|\\
 & \geq\frac{1}{2}\left|C\backslash\bigcup_{k=1}^{m}F_{k}D_{k}\right|
\end{align*}
(the second inequality follows from property \ref{enu:For-any-,-1}
of Definition \ref{def:inv. cond.}, the third by the $\epsilon$-disjointness
of the collection, and the last one by (\ref{eq:-1}), together with
the assumption $\epsilon\leq\frac{1}{2}$ ). Any element in $W_{2}$
appears as the left coordinate of $q$ different elements in $\bigcup_{k=1}^{m}C_{k}\times\{k\}$,
thus,
\begin{align*}
\left|\bigcup_{(d,n)\in\mathscr{D}}F_{n}d\right| & =\sum_{k=1}^{m}\left|F_{k}D_{k}\right|\\
 & \geq\frac{\epsilon}{5}\sum_{k=1}^{m}\left|C_{k}\right|\\
 & =\frac{\epsilon}{5}\left|\bigcup_{k=1}^{m}C_{k}\times\{k\}\right|\\
 & \geq\frac{\epsilon}{5}q\left|W_{2}\right|\\
 & \geq\frac{\epsilon}{10}q\left|C\backslash\bigcup_{k=1}^{m}F_{k}D_{k}\right|\\
 & \geq\frac{\epsilon^{2}}{10}q\left|C\right|
\end{align*}
assuming $q\geq\frac{20}{\epsilon^{2}}$, the Lemma is proved.
\end{proof}
\textbf{Proof of Theorem \ref{thm:For-any-}:} For any $x\in X$,
the number of fluctuations of $\mathsf{A}_{n}f(x)$ across $(\alpha,\beta)$
is equal to the number of fluctuations of $\mathsf{A}_{n}[f+||f||_{\infty}](x)$
across $(\alpha+||f||_{\infty},\beta+||f||_{\infty})$. Consequently,
for any $N$,
\[
D_{(F_{n}),f,\alpha,\beta,N}=D_{(F_{n}),f+||f||_{\infty},\alpha+||f||_{\infty},\beta+||f||_{\infty},N}.
\]
 Notice that $||f+||f||_{\infty}||_{\infty}\leq2||f||_{\infty}$,
and besides trivial cases, one has $0<\alpha+||f||_{\infty}$. Hence,
for any $S>0$ and $\alpha<\beta$, any estimate of $\mu\left(D_{N}\right)$,
where $D_{N}$ is defined w.r.t. any non negative function $0\leq f\leq2S$
and the gap $[\alpha+S,\beta+S]\subset(0,\infty)$, is an estimate
of $\mu\left(D_{N}\right)$, where $D_{N}$ is defined w.r.t. any
function $||f||_{\infty}\leq S$ and the gap $[\alpha,\beta]\subset\mathbb{R}$.
Thus from now on, we shall assume $0\leq f\leq S$ and $0<\alpha<\beta$.

Fix $x\in X$, $M\in\mathbb{N}$, and let $\Omega\subset G$ be a
set which is sufficiently invariant w.r.t. $\bigcup_{n=1}^{M}F_{n}$,
so that the set
\[
B=\left\{ g\in\Omega:\:\bigcup_{n=1}^{M}F_{n}g\subset\Omega\right\} 
\]
has size close to $|\Omega|$. We will give an upper bound to the
relative density $\frac{|C|}{|\Omega|}$, where
\[
C=C_{x,M}=\left\{ c\in B:\,cx\in D_{N,M}\right\} 
\]
 This upper bound won't depend on $x$ or $M$, and thus by the transference
principle, it will give an upper bound for $\mu\left(D_{N}\right)$,
as it is shown at the end of the proof. 

Take 
\[
\delta=\min\left\{ \frac{1}{2}\left(\frac{\beta}{\alpha}-1\right),\frac{1}{2}\right\} ,
\]
and choose $\frac{1}{4}>\epsilon>0$ small enough so that the following
three inequalities hold:
\begin{align}
\frac{(\beta-4\epsilon S)\left(1-\epsilon\right)}{\alpha} & \geq1+\delta>1\label{eq:-3}\\
(1-\epsilon)(1+\delta) & \geq(1+\delta/2)\label{eq:-2}\\
(1-\epsilon) & \geq(1+\delta/2)^{-1}\label{eq:-4}
\end{align}
Take $q\in\mathbb{N}$ and $0<\lambda\leq\epsilon/2$ so that the
conclusion of Lemma \ref{lem:Key1 Lemma} will take place with $\epsilon/2$.

The first step is to replace $C$ with a union of $\epsilon$-disjoint
collections of size not much less than $|C|$, where for each set
in the collection, the average of $f$ at $x$ on it is above $\beta$.
For that, use the first group of $q$ fluctuations to find for each
$c\in C$ an increasing sequence $n_{1}(c)<,...,<n_{q}(c)$ s.t. $\mathsf{A}_{n_{i}(c)}f(cx)\geq\beta$
for each $1\leq i\leq q$. Then, by applying Proposition \ref{lem:Key1 Lemma},
one take an $\epsilon$-disjoint collection $\left(F_{n}c\right)_{(c,n)\in\mathscr{B}_{1}}$,
where its union $C_{1}=\bigcup_{(c,n)\in\mathscr{B}_{1}}F_{n}c$ is
in $\Omega$ and of size $\left|C_{1}\right|\geq(1-\epsilon)\left|C\right|$.
The next step will be done recursively $(\lfloor\frac{N}{2q}\rfloor-1)$
times, thus we introduce it in a more general form:
\begin{lem}
\label{lem:inside proof}Let $f,\left(F_{n}\right),C,\alpha,\beta,\delta,\epsilon,N$
and $q$ be as above. Let $N_{k}\leq N-2q$, and suppose that $\mathscr{B}_{k}\subset C\times\mathbb{N}$
is a collection of tuples s.t. :

(i) For each $(c,n)\in\mathscr{B}_{k}$ the average $\mathsf{A}_{n}f(cx)$
is one of $cx$'s first $N_{k}$ upcrossings to above $\beta$.

(ii) the collection $\left(F_{n}c\right)_{(c,n)\in\mathscr{B}_{k}}$
is $\epsilon$-disjoint.

Then there exists a collection $\mathscr{B}_{k+1}\subset C\times\mathbb{N}$
of tuples s.t. :

(i) For each $(c,n)\in\mathscr{B}_{k+1}$, the average $\mathsf{A}_{n}f(cx)$
is one of $cx$'s first $N_{k}+2q$ upcrossings to above $\beta$.

(ii) The collection $\left(F_{n}c\right)_{(c,n)\in\mathscr{B}_{k+1}}$
is $\epsilon$-disjoint.

(iii) $\left|\bigcup_{(c,n)\in\mathscr{B}_{k+1}}F_{n}c\right|\geq(1+\delta/2)\left|\bigcup_{(c,n)\in\mathscr{B}_{k}}F_{n}c\right|$.
\end{lem}

\textbf{Proof of Lemma \ref{lem:inside proof}.} Denote $C_{k}=\bigcup_{(c,n)\in\mathscr{B}_{k}}F_{n}c$.
To each $g\in C_{k}$ we will associate a subsequence of $\left(F_{n}g\right)$
of length $q$, in order to apply Lemma \ref{lem:Key1 Lemma} to the
set $C_{k}$: For any $g\in C_{k}$, choose some $c=c(g)$ so that
$(c,n)\in\mathscr{B}_{k}$ for some $n$ and $g\in F_{n}c$. Associate
to $g$ the indices of the next $q$ downcrossings to below $\alpha$
of $c$, $n<n_{1}(c)<...<n_{q}(c)$. By Proposition \ref{lem:Key1 Lemma},
there is an $\nicefrac{\epsilon}{2}$-disjoint collection $\left(F_{n}g\right)_{(g,n)\in\mathscr{B}_{k}'}$
, with union $C_{k}'=\bigcup_{(g,n)\in\mathscr{B}_{k}'}F_{n}g\subset\Omega$
that satisfies one of the two options in the conclusion of Proposition
\ref{lem:Key1 Lemma}. Next, we define another index set $\mathscr{B}_{k}''$
to be
\[
\mathscr{B}_{k}''=\left\{ (c,n):\,\exists(g,n)\in\mathscr{B}_{k}',\,c(g)=c\right\} 
\]
and the union of its associated collection
\[
C_{k}''=\bigcup_{(c,n)\in\mathscr{B}_{k}''}F_{n}c.
\]
For any $(c,n)\in\mathscr{B}_{k}''$, let $(g,n)\in\mathscr{B}_{k}'$
be such that $c(g)=c$ . Then, $\left(F_{n}\right)$ being $\lambda$-good,
by \ref{enu:For-any-} of Definition \ref{def:inv. cond.}, 
\[
\left|F_{n}g\triangle F_{n}c(g)\right|<\lambda\left|F_{n}\right|\leq\epsilon/2\left|F_{n}\right|.
\]
 That, together with $\left(F_{n}g\right)_{(g,n)\in\mathscr{B}_{k}'}$
being $\epsilon/2$-disjoint, implies that 
\begin{equation}
\left(F_{n}c\right)_{(c,n)\in\mathscr{B}_{k}''}\text{ is \ensuremath{\epsilon}-disjoint}\label{eq:-6}
\end{equation}
and that
\begin{align}
|C_{k}''\cap C_{k}'| & \geq\sum_{(g,n)\in\mathscr{B}_{k}'}\left((1-\epsilon/2)|F_{n}g|-|F_{n}g\backslash F_{n}c(g)|\right)\label{eq:-5}\\
 & \geq(1-\epsilon)\sum_{(g,n)\in\mathscr{B}_{k}'}|F_{n}g|\nonumber \\
 & \geq(1-\epsilon)|C_{k}'|\nonumber 
\end{align}

This relation together with $C_{k}'$ being as in the conclusion of
Proposition \ref{lem:Key1 Lemma}, gives one of the following two
options:

1. Either $\left|C_{k}'\right|\geq2|C_{k}|$, in which case (\ref{eq:-5})
implies that 
\begin{equation}
|C_{k}''|\geq2(1-\epsilon)|C_{k}|\label{eq:-8}
\end{equation}
,

2. or $\left|C_{k}'\right|<2|C_{k}|$, but$\left|C_{k}'\bigcap C_{k}\right|\geq(1-\epsilon/2)|C_{k}|$,
which implies
\begin{align}
|C_{k}''\cap C_{k}| & \geq|C_{k}'\cap C_{k}|-|C_{k}'\backslash C_{k}''|\label{eq:-9}\\
 & \geq(1-\epsilon/2)|C_{k}|-\epsilon|C_{k}'|\nonumber \\
 & \geq(1-\epsilon/2)|C_{k}|-2\epsilon|C_{k}|\nonumber \\
 & >(1-3\epsilon)|C_{k}|\nonumber 
\end{align}

In both cases one can conclude that $\left|C_{k}''\right|\geq(1+\delta)|C_{k}|$:
for the first case (\ref{eq:-8}), $\epsilon<\frac{1}{4}$ and $\delta\leq\frac{1}{2}$
gives
\[
2(1-\epsilon)\geq1.5\geq1+\delta.
\]
 For the second case (\ref{eq:-9}), this can be observed by the next
calculation:

By (\ref{eq:-6}), there are pairwise disjoint sets $E''_{(n,c)}\subset F_{n}c$
(for each $(n,c)\in\mathscr{B}_{k}''$), with $\left|E''_{(n,c)}\right|\geq(1-\epsilon)\left|F_{n}c\right|$.
Thus
\begin{align*}
\sum_{g\in C_{k}''}f(gx) & \leq\sum_{(c,n)\in\mathscr{B}_{k}''}\,\sum_{g\in F_{n}c}f(gx)\\
 & \leq\left(\sum_{\mathscr{B}_{k}''}\left|F_{n}c\right|\right)\alpha\\
 & \leq\frac{1}{1-\epsilon}\left(\sum_{\mathscr{B}_{k}''}\left|E''_{(n,c)}\right|\right)\alpha\\
 & \leq\frac{1}{1-\epsilon}\left|C_{k}''\right|\alpha
\end{align*}
On the other hand, the collection $\left(F_{n}c\right)_{(c,n)\in\mathscr{B}_{k}}$
is $\nicefrac{\epsilon}{2}$-disjoint, and so, there are pairwise
disjoint sets $E_{(n,c)}\subset F_{n}c$ (for each $(n,c)\in\mathscr{B}_{k}$),
with $\left|E_{(n,c)}\right|\geq(1-\nicefrac{\epsilon}{2})\left|F_{n}c\right|$.
Thus
\begin{align*}
\sum_{g\in C_{k}}f(gx) & \geq\sum_{(c,n)\in\mathscr{B}_{k}}\sum_{g\in E_{(n,c)}}f(gx)\\
 & =\sum_{\mathscr{B}_{k}}\left(\sum_{F_{n}c}f(gx)-\sum_{F_{n}c\backslash E_{(n,c)}}f(gx)\right)\\
 & \geq\sum_{\mathscr{B}_{k}}\left|F_{n}c\right|\left(\beta-\frac{\epsilon}{2}S\right)\\
 & \geq\left|C_{k}\right|\left(\beta-\frac{\epsilon}{2}S\right)
\end{align*}
if $\left|C_{k}''\bigcap C_{k}\right|\geq(1-3\epsilon)|C_{k}|$ as
in (\ref{eq:-9}), then:
\begin{align*}
\left|C_{k}\right|\left(\beta-\frac{\epsilon}{2}S\right) & \leq\sum_{g\in C_{k}}f(gx)\\
 & \leq\left|C_{k}\right|3\epsilon S+\sum_{g\in C_{k}''}f(gx)\\
 & \leq\left|C_{k}\right|3\epsilon S+\frac{1}{1-\epsilon}\left|C_{k}''\right|\alpha
\end{align*}
 Thus, with our choice of $\epsilon$ w.r.t. $\delta$ (\ref{eq:-3}),
we get that:
\[
\left|C_{k}''\right|\geq(1+\delta)|C_{k}|.
\]

In the same manner we constructed $\mathscr{B}_{k}''$, we use the
next $q$ upcrossings to above $\beta$ to construct a collection
$\mathscr{B}_{k+1}$ s.t. $\left(F_{n}c\right)_{\mathscr{B}_{k+1}}$
is an $\epsilon$-disjoint collection of upcrossings, with union $C_{k+1}=\bigcup_{\mathscr{B}_{k+1}}F_{n}g$
in $\Omega$ that satisfies one of the two options in the conclusion
of Proposition \ref{lem:Key1 Lemma}. In particular, we have:
\begin{align*}
\left|C_{k+1}\right| & \geq(1-\epsilon)\left|C_{k}''\right|\\
 & \geq(1-\epsilon)(1+\delta)\left|C_{k}\right|\\
 & \geq(1+\delta/2)\left|C_{k}\right|
\end{align*}

(the last inequality follows from the assumption $(1-\epsilon)(1+\delta)\geq(1+\delta/2)$),
and Lemma \ref{lem:inside proof} is proved.$\hfill\square$

Back to the proof of Theorem \ref{thm:For-any-}, from Lamma \ref{lem:inside proof}
it follows that there exist finite subsets of $\Omega,$ $C_{1},...,C_{\lfloor\frac{N}{2q}\rfloor}$
s.t. 
\begin{align*}
\left|\Omega\right| & \geq\left|C_{\lfloor\frac{N}{2q}\rfloor}\right|\geq(1+\delta/2)^{\lfloor\frac{N}{2q}\rfloor-1}\left|C_{1}\right|\\
 & \geq(1+\delta/2)^{\lfloor\frac{N}{2q}\rfloor-1}(1-\epsilon)\left|C\right|\\
 & \geq(1+\delta/2)^{\frac{N}{2q}-3}\left|C\right|
\end{align*}
(the last inequality follows partialy from the assumption $(1-\epsilon)\geq(1+\delta/2)^{-1}$).
Since
\[
\mu(D_{N,M})=\frac{1}{|\Omega|}\int\sum_{g\in\Omega}\mathbf{1}_{D_{N,M}}(gx)d\mu(x)\leq\int\frac{\left|C_{x,M}\right|}{\left|\Omega\right|}d\mu(x)+(1-\frac{\left|B\right|}{\left|\Omega\right|})
\]
 where $(1-\frac{\left|B\right|}{\left|\Omega\right|})$ can be made
arbitrarily small (by taking $\Omega$ to be arbitrarily invariant),
one have
\[
\mu(D_{N,M})\leq\int\frac{\left|C_{x,M}\right|}{\left|\Omega\right|}d\mu(x)\leq(1+\delta/2)^{-\left(\frac{N}{2q}-3\right)}
\]
Thus the claim of the theorem takes place with $c_{0}=(1+\delta/2)^{-\frac{1}{2q}},c_{1}=(1+\delta/2)^{3}$.$\hfill\square$

\lyxaddress{uriel.gabor@gmail.com}

\lyxaddress{Einstein Institute of Mathematics}

\lyxaddress{The Hebrew University of Jerusalem}

\lyxaddress{Edmond J. Safra Campus, Jerusalem, 91904, Israel}
\end{document}